\documentclass[oneside,a4paper,11pt]{article}
\usepackage{mathtools}
\usepackage{amsmath}
\allowdisplaybreaks[4]
\usepackage{csquotes}
\usepackage{amsthm}
\usepackage{amssymb}
\usepackage{mathrsfs}
\usepackage{color}
\usepackage{bm}
\usepackage{fancyhdr}
\usepackage{geometry}
\usepackage{hyperref}
\usepackage{enumerate}
\usepackage{graphicx}
\usepackage{subfig}
\usepackage{float}

\theoremstyle{definition}
\newtheorem{thm}{Theorem}[section]
\newtheorem{cor}[thm]{Corollary}
\newtheorem{lem}[thm]{Lemma}

\newtheorem{defn}[thm]{Definition}
\newtheorem{rem}[thm]{Remark}

\numberwithin{equation}{section}
\newcommand{\ddt}[1]{\frac{\mathrm{d}#1}{\mathrm{d}t}}
\newcommand{\pp}[2]{\frac{\partial#1}{\partial#2}}

\title{Combinatorial Calabi flows for ideal circle patterns in spherical background geometry}
\author{
        Ziping Lei~~
	Puchun Zhou
}

\date{}
\usepackage{fancyhdr}
\providecommand{\classification}[1]
{
	\small	
	\textbf{Mathematics Subject Classification (2020):} #1
}

\begin{document}
	\maketitle
	\begin{abstract}
 
 Combinatorial Calabi flows are introduced by Ge in his Ph.D. thesis (Combinatorial methods and geometric equations, Peking University, Beijing, 2012), and have been studied extensively in Euclidean and hyperbolic background geometry. In this paper, we introduce the combinatorial Calabi flow in spherical background geometry for finding ideal circle patterns with prescribed total geodesic curvatures.
 We prove that the solution of combinatorial Calabi flow exists for all time and converges if and only if there exists an ideal circle pattern with prescribed total geodesic curvatures. We also show that if it converges, it will converge exponentially fast to the desired metric, which provides an effective algorithm to find certain ideal circle patterns. To our knowledge, it is the first combinatorial Calabi flow in spherical background geometry.
\\
		\classification{52C26, 51M10, 57M50
		}	
		
	\end{abstract}
	
\section{Introduction} \label{sec:1}
Finding a canonical metric with prescribed curvature on a given manifold is an important problem in differential geometry. In order to find canonical metrics, geometric flows are introduced, and play more and more important role in differential geometry. To understand the geometry of 3-manifolds, Hamilton introduced the Ricci flow \cite{hamilton1982three} which has been used to prove Poincar\'{e} conjecture. For seeking metrics with constant curvature, Calabi studied the variational problem of ``Calabi energy'' and introduced the Calabi flow in his work \cite{calabi1,calabi2}.

Geometric flows can be also used to study metrics on polyhedral manifolds. In 2-dimensional case, Thurston studied circle packing metrics in his famous notes \cite{thurston1976geometry}, which might have cone singularities on vertices of polyhedral surfaces. Inspired by Ricci flows on manifolds, Chow and Luo \cite{MR2015261} introduced combinatorial Ricci flows for finding polyhedral surfaces without singularities, given by
\[\ddt{r_v}=-(2\pi-\alpha_v)s(r_v),~~\forall v\in V,\]
where $s(x)=x~(\sin{x}~\text{or}~ \sinh{x})$ for Euclidean (spherical or hyperbolic) background geometry. $\alpha_v$ is the cone angle of the singularity at $v$, and $r_v$ is the radius of the circle centered at $v$. $2\pi-\alpha_v$ is called the discrete Gaussian curvature at $v$ in the literature.
They proved the long time existence and the convergence of the flow under some combinatorial conditions in Euclidean and hyperbolic background geometry. For other applications of combinatorial Ricci flows, we refer the readers to \cite{ge2021combinatorial,feng2022combinatorial}.

Moreover, Ge in his thesis \cite{ge_phd} first introduced the combinatorial Calabi flow for Thurston's circle packing metrics. Ge \cite{ge_phd,Ge2017Combinatorial}, Ge and Hua \cite{gehua2018combinatorial}, and Ge and Xu \cite{ge20162} proved that combinatorial Calabi flows exist for all time and converge exponentially fast to Thurston’s circle packing on surfaces if and only if there exists a circle packing metric of constant (zero resp.) discrete Gaussian curvatures in Euclidean (hyperbolic resp.) background geometry. 
 After that, so-called combinatorial p-th Calabi flows are also studied in \cite{lin2019combinatorial,feng2020combinatorial}, which generalize the work in \cite{Ge2017Combinatorial}.

Those work mentioned above mainly study combinatorial Calabi flows on polyhedral surfaces in Euclidean and hyperbolic background geometry.
However, in spherical background geometry, there is no result of combinatorial Calabi flows as far as we know. This is because the variational principle of discrete Gaussian curvatures in spherical background geometry is much more complicated than Euclidean and hyperbolic cases. Nevertheless, recently Nie \cite{nie2023circle} introduced a new variational principle for the total geodesic curvatures of ideal circle patterns in spherical background geometry. He proved the existence and uniqueness of ideal circle patterns with prescribed total geodesic curvatures. 

Inspired by his work, we introduce combinatorial Calabi flows in spherical background geometry to find ideal circle patterns with prescribed
total geodesic curvatures. Moreover, we prove that the combinatorial Calabi flow exists
for all time and it converges if and only if the prescribed total geodesic curvatures satisfy a combinatorial condition.

    \section{Preliminaries}\label{sec:2}
    \subsection{Ideal circle patterns in spherical background geometry}
Circle patterns are useful tools for studying the geometry and topology of 3-manifolds, which were rediscovered by Thurston in \cite{thurston1976geometry}. The theorem of existence and uniqueness of circle patterns on spheres was proved by Andreev in \cite{andreev1970convex1} and \cite{andreev1970convex2}, which is called the Koebe-Andreev-Thurston theorem. After that, Colin de Verdi\`{e}res proposed a variational principle which links circle patterns to minimizers of certain potential functions, see \cite{MR1106755}.

In particular, we consider an ideal convex polyhedron $P$ in 3-dimensional hyperbolic
space $\mathbb{H}^3$, i.e. the convex polyhedron whose vertices locate at infinity. In the Poincar\'{e}
ball model  $\mathbb{B}^3$ of  $\mathbb{H}^3$, each face of $P$ corresponds to a half sphere which intersects with $\partial  \mathbb{B}^3$ vertically at a spherical circle. Those circles form a circle pattern which is called ideal circle pattern. Two faces of $P$ with a common edge $e$ and dihedral angle $\Phi(e)$ correspond to two circles on $\partial \mathbb{B}^3$ with intersection angles $\Phi(e)$. So it is a natural problem to find ideal circle patterns with prescribed combinatorial type on $\partial\mathbb{B}^3$ with intersection angles $\Phi\in(0,\pi)^E$. This problem has been already resolved by Rivin in \cite{MR1193599}. He actually characterized ideal polyhedra in hyperbolic 3-space. 

Ideal circle patterns in Euclidean and hyperbolic background geometry had been systematically studied by Ge, Hua and Zhou \cite{ge2021combinatorial}. In their work, they extended Thurston's circle packing theory to ideal circle patterns and proved the existence and rigidity of ideal circle patterns with the help of combinatorial Ricci flows. Moreover, they gave a new proof of Rivin's elegant theorem \cite{MR1370757} with flow approach.
It is worth noting that Bobenko and Springborn \cite{MR2022715} and Guo \cite{guo2007note} also got some related results.


In our work, we consider the ideal circle patterns on surfaces with spherical conical metrics, which will be defined in this section. It is a generalization of ideal circle patterns obtained from ideal polyhedra in $\mathbb{H}^3$. 

We first define a special type of graph embedding in closed surfaces. Let $\Sigma$ be a closed surface and $G=(V,E)$ be a finite graph without loop. Let $\eta:V\cup E\rightarrow\Sigma$ be a graph embedding. A face is a connected component of $\Sigma\backslash\eta(V\cup E).$ By $F$ we denote the set of faces. We will not distinguish $\eta(V\cup E)$ with $V\cup E.$
We call an embedding $\eta$ a \textit{closed 2-cell embedding} if the following hold, see \cite{barnette1987generating}.
\begin{enumerate}
    \item The closure of every face is homeomorphic to a closed disk.
    \item Any face is bounded by a simple closed curve consists of finite many edges. 
\end{enumerate}
Note that the definition above allows multiple edges and face whose boundary consists of 2 edges.

By $v < e$ ($e < f$ resp.) we mean that a vertex $v$ (an edge e resp.)
is incident to an edge $e$ (a face $f$ resp.). By $v\sim w$, we mean $e=\{v,w\}\in E$. For set $W\subset V,$ by $E(W)$ we denote the set of edges that have at least one end point within $W,$ i.e.
\[E(W) = \{e\in E: \exists v\in W, ~s.t.~~ v<e\}.\]
Let $\Phi=\{\Phi(e)\}_{e\in E} \in (0,\frac{\pi}{2}]^E$ be intersection angles on edges and $r=\{r_v\}_{v\in V}\in (0,\frac{\pi}{2})^V$ be radii defined on vertices. 
We will construct an ideal circle pattern in spherical background geometry with respect to the closed 2-cell embedding $\eta$, intersection angles $\Phi$ and radii $r$.

For each face $f$, we add an auxiliary vertex $v_f$ in the
interior of the face. We define the \textit{incidence graph} of $G$ in the following way.
\begin{defn}
    An \textit{incidence graph} $I(G)$ is a bipartite graph with the bipartition $\{V,V_F\}$. For $v\in V$ and $v_f\in V_F$, $v$ and $v_f$ are adjacent in $I(G)$ if and only if $v$ is on the boundary of the face $f$.
\end{defn}
The notion of incidence graph is used to demonstrate the relationship between vertices and faces; see e.g.\ \cite{coxeter1950self}.
Let $e=\{v,w\}\in E$. Assume that $f_1,f_2\in F$ are faces whose boundaries contain $e$. By $Q_e$ we denote the quadrilateral $vv_{f_1}wv_{f_2}$ in the incidence graph $I(G)$. 
Obviously, any edge in the original graph $G$ uniquely corresponds to a quadrilateral in the incident graph $I(G)$, as shown in the Figure \ref{circle_pattern}. 

\begin{figure}[H]
		\centering
		\includegraphics[width=3in]{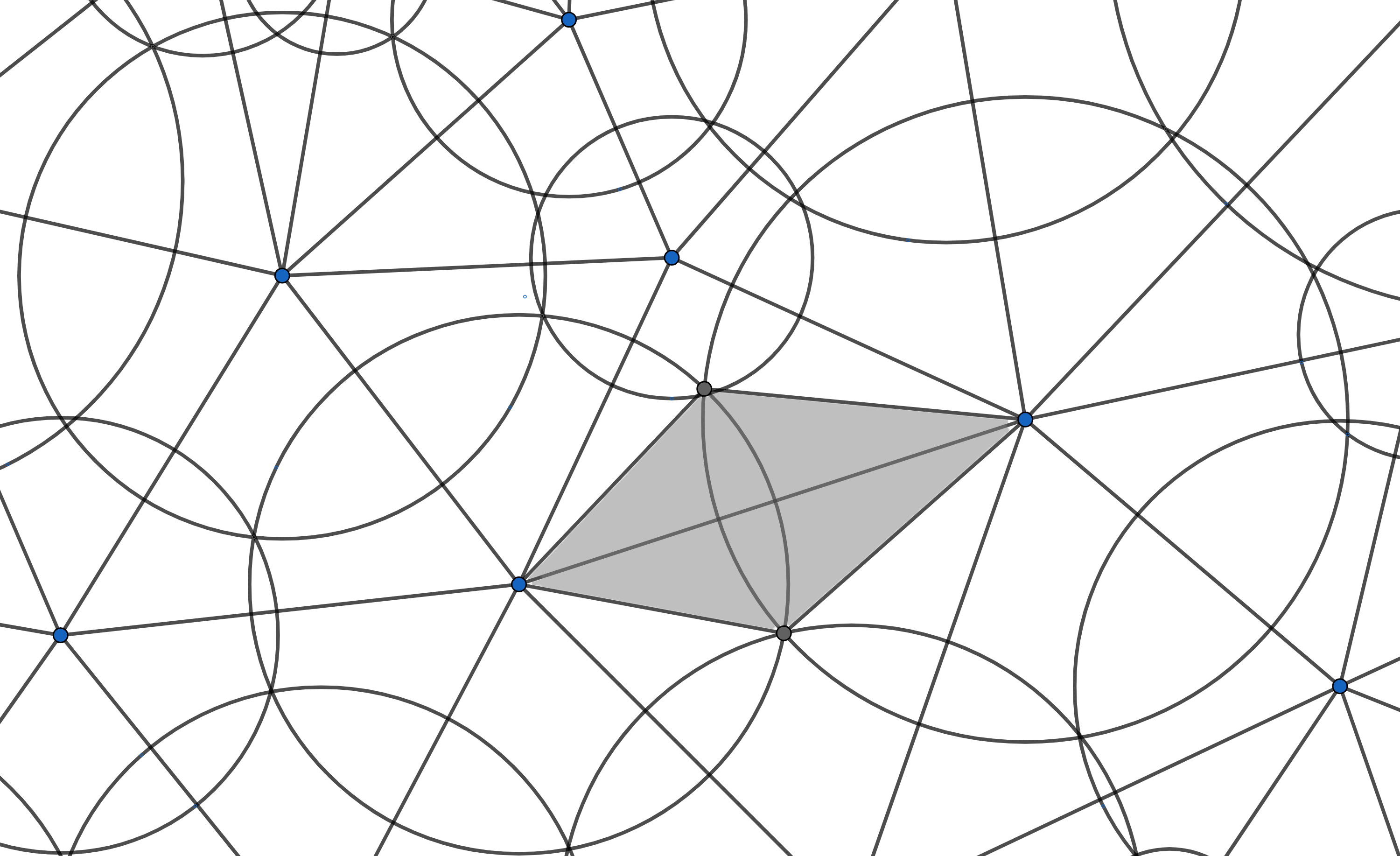}
		\caption{Local part of an ideal circle pattern.}
		\label{circle_pattern}
\end{figure}

Then we can construct a spherical quadrilateral $\tilde{Q}_e$ with 
\begin{align*}
\angle vv_{f_1}w=\angle vv_{f_2}w=\pi-\Phi(e),~
|vv_{f_1}|=|vv_{f_2}|=r_v,~ |wv_{f_1}|=|wv_{f_2}|=r_w,
\end{align*}
as in Figure \ref{quadrilateral}.

\begin{figure}[H]
		\centering
		\includegraphics[width=3in]{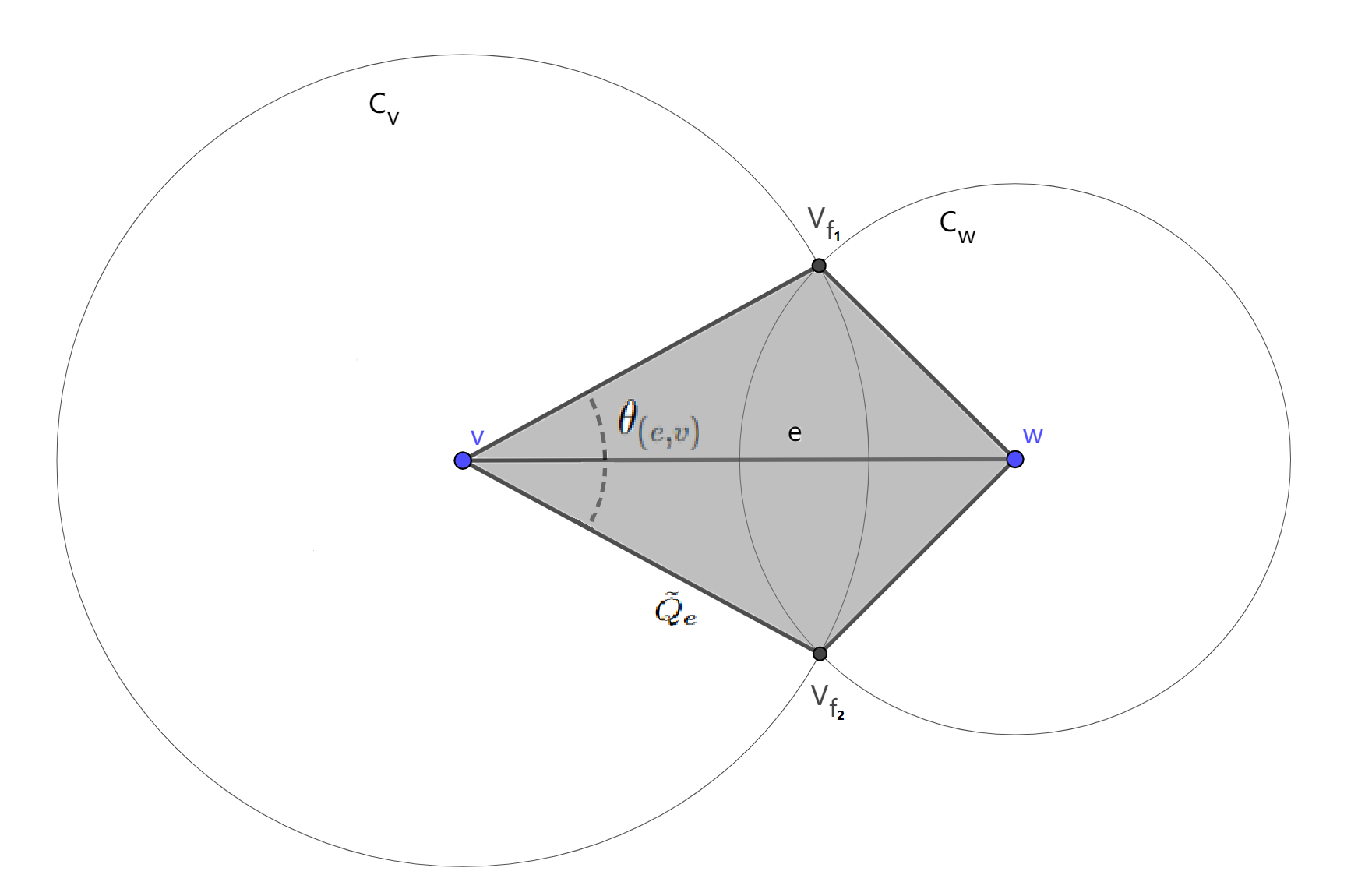}
		\caption{Spherical quadrilateral~$\tilde{Q}_e.$}
		\label{quadrilateral}
\end{figure}
By $\theta_{(e,v)}\in (0,2\pi)$ we denote $\angle v_{f_1}vv_{f_2}$ in $\Tilde{Q}_e.$ 
Gluing all the spherical quadrilaterals along the edges in $I(G)$, we obtain a metric $g=g(\eta,r,\Phi)$ on $\Sigma$ (for the gluing procedure, see \cite[Chapter 3]{burago2022course}). It is clear that $g$ is a metric of constant curvature outside $V\cup V_F.$ $V$ and $V_F$ are possible cone points in $\Sigma.$ For $v\in V$, the cone angle $\alpha_v$ is given by
\[\alpha_v = \sum_{e:v<e}\theta_{(e,v)}.\]
For $f\in F$, the cone angle at $v_f$ is given by
\[\alpha_{v_f}=\sum_{e:e<f}(\pi-\Phi(e)).\]
\begin{defn}
Given $\eta,r,\Phi$ and $g=g(\eta,r,\Phi)$ described as above,
the \textit{ideal circle pattern} $\mathcal{P}=\{C_v\}_{v\in V}$ with respect to $\eta,r$ and $\Phi$ is a collection of circles such that $C_v$ is centered at $v$ with radius $r_v$ in $(\Sigma,g)$ for each $v\in V.$
\end{defn}
\subsection{Problem of prescribed total geodesic curvatures}
For $v\in V$, by $k_v$ and $l_v$ we denote the geodesic curvature and circumference of $C_v,$ which is given by 
\[k_v = \cot r_v, ~l_v = \alpha_v\sin r_v.\]
Integrate the geodesic curvature along $C_v$, we get the \textit{total geodesic curvature} of $C_v$, which is denoted by $L_v,$ with
\[L_v = k_vl_v=\alpha_v\cos r_v.\]
Moreover, assume that $v<e$, by $L_{(e,v)}$ we denote the total geodesic curvature of $C_v\cap \tilde{Q}_e$, i.e.
\[L_{(e,v)}=\theta_{(e,v)}\cos{r_v}.\]
So for ideal circle patterns, we have
\begin{align}
    L_v = \sum_{e:v<e}L_{(e,v)}.
\end{align}
Fixed $\eta$ and $\Phi$, the total geodesic curvatures of an ideal circle pattern are uniquely determined by radii $r\in (0,\frac{\pi}{2})^V.$ So we can write total geodesic curvatures as functions of $r$, i.e.
\[L_v = L_v(r),~~~\forall v\in V.\]
Given $\Phi$ and prescribed total geodesic curvatures $\{\hat{L}_v\}_{v\in V}\in (0,\infty)^V$, an interesting problem raised by Nie in \cite{nie2023circle} is whether there are some radii $r$, such that
\[L_v(r) = \hat{L}_v,~~~\forall v\in V.\]
In his article, Nie found the necessary and sufficient condition of the existence of ideal circle patterns in spherical background geometry with prescribed total geodesic curvatures.
\begin{thm}[Nie]\label{nie}
Given $G = (V,E)$ and $\Phi\in (0,\frac{\pi}{2}]^E,$ for prescribed total geodesic curvatures
$\{\hat{L}_v\}_{v\in V},$ there exists radii $r$ such that 
\[ L_v(r)=\hat{L}_v, ~~~\forall v\in V, \]
if and only if 
$\{\hat{L}_v\}_{v\in V}$ satisfy the condition
\begin{align}\label{condition}
    \sum_{v\in W}\hat{L}_v<2\sum_{e\in E(W)}\Phi(e),~~~\forall W\subset V.
\end{align}
For seeking desired radii $r$, we introduced a combinatorial curvature flow  in our recent work, see \cite{ge2023combinatorial}, i.e. 
\begin{align}
\label{curvature_flow}
    \ddt{r_v}=\frac{L_v-\hat{L}_v}{2}\sin2r_v,~~~\forall v\in V.
\end{align}

\end{thm}
We proved the long time existence and convergence of the flow under the condition \eqref{condition}, which can be stated as follow.
\begin{thm}
    Given a 2-cell embedding of $G=(V,E)$ in a closed surface $\Sigma,$ let $\Phi\in (0,\frac{\pi}{2}]^E$ be intersection angles and $\{\hat{L}_v\}_{v\in V}$ be prescribed total geodesic curvatures. The flow \eqref{curvature_flow} exists for all time. Moreover, it converges if and only if the condition \eqref{condition} holds.\label{Thm 3}
\end{thm}

\subsection{Combinatorial Calabi flows with prescribed total geodesic curvatures}

Now we introduce combinatorial Calabi flows for ideal circle patterns in spherical background geometry. For smooth surfaces, we know the smooth Calabi flow is defined as
\[\frac{\partial g}{\partial t}=\Delta R g,\]
where $R$ is the Gaussian curvature, $\Delta$ is the Laplace-Beltrami operator.

For combinatorial Calabi flows in spherical background geometry, we have to define the discrete Laplace operator, which is an analogue of the Laplace-Beltrami operator on smooth surfaces.

For $r=\{r_v\}_{v\in V}$, we suppose that $|V|=n$ if there is no specific explanation. All vertices, marked by $v_1,...,v_n$, are ordered one by one and we often write
$i$ instead of $v_i$ if there is no confusion. We set
\begin{align}
K_i=\ln \cot r_i,\label{coordinate}
\end{align}
where $i=1,2,...,n$. The coordinate transformation maps $r\in (0,\frac{\pi}{2})^n
 $ to $K\in\mathbb{R}^n$ correspondingly. We denote this coordinate
transformation as a map
\[\ln\cot : (0,\frac{\pi}{2})^n \rightarrow \mathbb{R}^n, r=(r_1,...r_n) \mapsto K=(K_1,...,K_n)=\ln \cot r.\]

We first consider the Calabi energy with respect to total geodesic curvatures of ideal circle patterns in spherical background geometry, i.e.
	 \begin{align}
	 	\mathcal{C}(r) = \frac{1}{2}\|L\|^2\label{calabi},
	 \end{align}
 	 where 	 \[L = (L_1,...,L_{n})^T\in(0,\infty)^{n}~\] are total geodesic curvatures of the ideal circle pattern.
 	 
 	 For prescribed total geodesic curvatures 
 	 $$\hat{L}=(\hat{L}_1,...,\hat{L}_{n})^T\in(0,\infty)^{n},$$ we also have the Calabi energy, i.e.
 	 \begin{align}
 	 	\bar{\mathcal{C}}(r)=\frac{1}{2}\|L-\hat{L}\|^2.\label{precalabi}
 	 \end{align}
 	   
For $L=(L_1,...,L_n)^T$, $L$ is a function of the variable $K$. Then the negative gradient flow of $\mathcal{C}(K)=\frac{1}{2}\|L\|^2$ is defined as
 	 \begin{align}
 	 	\ddt{K}=-\begin{pmatrix}
 	 		\frac{\partial L_1}{\partial K_1}&\cdots&\frac{\partial L_{n}}{\partial K_{1}}\\
 	 		\vdots&\ddots&\vdots\\
 	 		\frac{\partial L_1}{\partial K_{n}}&\cdots&\frac{\partial L_{n}}{\partial K_{n}}\\
 	 	\end{pmatrix}L=-J^T L,\label{calabi_flow}
 	 \end{align}
where $J$ is the Jacobi matrix, i.e.

\[J = \frac{\partial(L_1,...,L_{n})}{\partial(K_1,...,K_{n})}=\begin{pmatrix}
 	 		\frac{\partial L_1}{\partial K_1}&\cdots&\frac{\partial L_{1}}{\partial K_{n}}\\
 	 		\vdots&\ddots&\vdots\\
 	 		\frac{\partial L_n}{\partial K_{1}}&\cdots&\frac{\partial L_{n}}{\partial K_{n}}\\
\end{pmatrix}.\]
We can prove that $J$ is symmetric and strictly diagonally dominant in section \ref{sec:5}.

So we can rewrite the flow \eqref{calabi_flow} as
\begin{align}
	\ddt{K}=\Delta L=-J^TL.\label{flow}
\end{align}
 
$\Delta$ is a negative definite operator and is similar to the classical discrete Laplace operator, see \cite[Chapter 1]{chung1997spectral}.
So we have the combinatorial Calabi flow for the prescribed total geodesic curvatures of ideal circle patterns in spherical background geometry as follow.
\begin{defn}
For given prescribed total geodesic curvature $\hat{L} \in(0,\infty)^V,$ the \textit{combinatorial Calabi flow} in spherical background geometry is
\begin{align}
\ddt{K}=\Delta(L-\hat{L})=-J^T(L-\hat{L}).\label{preflow}
\end{align}
\end{defn}

\subsection{Main results}\label{sec:3}
For the combinatorial Calabi flow with prescribed total geodesic curvatures, we have the main theorem as follow.
\begin{thm} \label{main}
    Given a closed 2-cell embedding $\eta$ and intersection angles $\Phi\in (0,\frac{\pi}{2}]^V,$ the combinatorial Calabi flow \eqref{preflow} with prescribed total geodesic curvatures $\{\hat{L}_v\}_{v\in V}$ exists for all time, and the following statements are equivalent:
    \begin{enumerate}[i.]
        \item The combinatorial Calabi flow \eqref{preflow} converges for any initial data.
        \item The prescribed curvature flow \eqref{curvature_flow} converges for any initial data.
        \item$\{\hat{L}_v\}_{v\in V}$ satisfies \eqref{condition}.
    \end{enumerate}
        Moreover, if the flow \eqref{preflow} converges, then it converges exponentially fast to an unique ideal circle pattern with $L_v=\hat{L}_v$ for each vertex $v.$ 
\end{thm}

\section{The potential function for total geodesic curvatures}\label{sec:5}
In this section, we briefly recall the variational principle of total geodesic curvatures of ideal circle patterns in spherical background geometry, which was observed by Nie in \cite{nie2023circle}.

\begin{lem}(Nie)\label{vari}
Set $\eta,\Phi$ as in section \ref{sec:2} and $e=\{v,w\}$, then we have
\[\pp{L_{(e,v)}}{K_w}=\pp{L_{(e,w)}}{K_v}.\]
Moreover, 
\[\pp{L_{(e,v)}}{K_w}<0,\pp{L_{(e,v)}}{K_v}>0,\pp{(L_{(e,v)}+L_{(e,w)})}{K_w}>0.\]
Hence, $\pp{(L_{(e,v)},L_{(e,w)})}{(K_v,K_w)}$ is strictly diagonally dominant.
\end{lem}
\begin{proof}
    This can be directly derived from the cotangent 4-part formula in spherical background geometry i.e.
\begin{align}
    \cot\frac{\theta_{(e,v)}}{2}=\frac{1}{\sin\Phi(e)}(\cot r_w\sin r_v+\cos r_v\cos\Phi(e)).\label{cot}
\end{align}
    By differentiating both sides of \eqref{cot} with respect to $K_v=\ln\cot r_v$ and $K_w=\ln\cot r_w$, and with the help of sine law, i.e.
\[\frac{\sin{\frac{\theta_{(e,v)}}{2}}}{\sin{r_w}}=\frac{\sin{\frac{\theta_{(e,w)}}{2}}}{\sin{r_v}},\]
    a tedious calculation gives 
    \begin{align}
        &\pp{L_{(e,v)}}{K_w}=-\frac{2\cos r_v\cos r_w\sin\frac{\theta_{(e,v)}}{2}\sin\frac{\theta_{(e,w)}}{2}}{\sin{\Phi(e)}}<0,\label{calculate1}\\
    &\pp{(L_{(e,v)}+L_{(e,w)})}{K_v}=\sin^2r_v\cos r_v(\theta_{(e,v)}-\sin\theta_{(e,v)})>0.\label{calculate2}
    \end{align}
    More details can be found in \cite{ge2023combinatorial}, here we omit the proof.
\end{proof}
\begin{rem}
    In \cite{nie2023circle}, Nie gives a geometric interpretation of this lemma.
\end{rem}

For an edge $e=\{v,w\}$, one can define an 1-form 
$$\omega_e=L_{(e,v)}\mathrm{d}K_v+L_{(e,w)}\mathrm{d}K_w.$$
By Lemma \ref{vari}, it is a closed form. So it is easy to see
\[\sum_{e\in E}\omega_e=\sum_{v\in V}\sum_{e:v<e}L_{(e,v)}\mathrm{d}K_v=\sum_{v\in V}L_v\mathrm{d}K_v\]
is a closed form. Since $\sum_i(L_i-\hat{L}_i)\mathrm{d}K_i$ is also a closed form, its integral is well-defined. Therefore, for prescribed total geodesic curvatures $\{\hat{L}_i\}_{i=1}^n,$ we can define a potential function, given by
\[\mathcal{E}(K_1,K_2,...,K_n)=
\int^{(K_1,...,K_n)}\sum_i(L_i-\hat{L}_i)\mathrm{d}\tau_i.\]
It is clear to see that 
\[\nabla_K\mathcal{E}=(L_1-\hat{L}_1,...,L_n-\hat{L}_n)^T.\]
Therefore, the Hessian of $\mathcal{E}$ equals to $J^T.$ 
Since
\begin{align*}
    J_{vv}&=\sum_{e:v<e}\pp{L_{(e,v)}}{K_v}>0,~~~\forall v\in V,\\
    J_{vw}&=\pp{L_v}{K_w}<0,~~~\forall w\sim v ,\\
    J_{vw}&=0,~~~\forall w\nsim v,
\end{align*}
$J$ is symmetric and strictly diagonally dominant by Lemma \ref{vari}. Consequently, $\mathcal{E}$ is strictly convex.
The potential function $\mathcal{E}$ is useful for the proof of our main theorem. Since the function $\mathcal{E}$ is defined in $\mathbb{R}^n$ and is strictly convex, it is proper if it attains a critical point. Therefore, we have the following corollary.
\begin{cor}\label{proper}
    For given prescribed total geodesic curvatures $\hat{L}$, if there exists $\overline{K}\in \mathbb{R}^n$ such that $L(\overline{K})=\hat{L}$, then the potential function $\mathcal{E}$ is proper, and it attains a global minimum at $\overline{K}$. 
\end{cor}

\section{Long time existence}\label{sec:4}
	Now we prove the long time existence of the combinatorial Calabi flow \eqref{preflow}.
\begin{thm}[long time existence]
	Given an initial data $K_0\in \mathbb{R}^{n},$ the flow $K(t)$ of \eqref{preflow} exists for all time $t\in \mathbb{R}.$
\end{thm}
\begin{proof}
	Note that $-J^T(L-\hat{L})$ is a smooth function of $K$. By Picard's theorem, we have the short time existence and uniqueness of the flow for any initial data.

 Now we turn to estimate the 2-norm of the matrix $J^T$. By $\lambda_{max}$ we denote the maximum eigenvalue of $J^T$. Due to the Gershgorin's theorem, we have
	\begin{align*}
        \|J^T\|_2=\lambda_{max}&\le \max_{v\in V}\sum_{w\in V}|J_{vw}|.
	\end{align*}
For each vertex $v$, by $d_v$ we denote the degree of $v$, i.e. the number of edges whose endpoints contain $v$. By \eqref{calculate1} and \eqref{calculate2}, for each vertex $v$, we have
\begin{align*}
&\sum_{w\in V}|J_{vw}| \le|J_{vv}-\sum_{w\sim v}J_{vw}|+2\sum_{w\sim v}|J_{vw}|,
\\
&\le 2\sin^2r_v\cos r_v(\theta_{(e,v)}-\sin\theta_{(e,v)})
+\sum_{e:v<e,e=\{v,w\}}\frac{2\cos r_v\cos r_w\sin\frac{\theta_{(e,v)}}{2}\sin\frac{\theta_{(e,w)}}{2}}{\sin{\Phi(e)}}
,\\
&\le4d_v\pi+ 2\sum_{e:v<e}\frac{1}{\sin\Phi(e)}.
\end{align*}
Since $\theta_{(e,v)}\in(0,2\pi)$ for any edge $e$, 
        we have 
        \[0<L_v=\sum_{e:v<e}\theta_{(e,v)}\cos r_v\le 2d_v\pi,~~\forall v\in V.\]
Therefore we have 
\begin{align}
\|\ddt{K}\|&\le\|J^T\|_2\|L-\hat{L}\|\\
&\le 4\sqrt{|V|}\max_v\{d_v\pi+\sum_{e:v<e}\frac{1}{\sin\Phi(e)}\}\max_v\{2d_v\pi+\hat{L}_v\}.\label{estimate}
\end{align}

	Therefore $\|\ddt{K}\|$ is bounded by a constant which only depends on the structure of the graph and intersection angles $\Phi.$ Then by the classical theorem of ordinary differential equations, the flow exists for all time $t\in\mathbb{R}.$
\end{proof}

\section{The proof of the main theorem}\label{sec:6}

In this section, we will prove Theorem \ref{main}. Due to Theorem \ref{Thm 3}, we only need to prove that i and iii are equivalent.
\begin{proof}
    For \enquote{i$\Rightarrow$iii}, suppose the combinatorial Calabi flow $K(t)$ converges to $K^*\in \mathbb{R}^n$ for some initial data, the sequence $\{K(l)\}_{l\in\mathbb{N}}$ also converges to $K^*.$ Since the Calabi energy $\bar{\mathcal{C}}(K(t)) = \frac{1}{2}\|L(K(t))-\hat{L}\|^2$ is a smooth function of $K$, it also converges. By the mean value theorem, we have 
    \begin{align*}
    \bar{\mathcal{C}}(K(l+1))-\bar{\mathcal{C}}(K(l))=\nabla \bar{\mathcal{C}}(K(\xi_l))^T\cdot \ddt{K}(\xi_l)=-\|\nabla \bar{\mathcal{C}}(K(\xi_l))\|^2,~~l\in \mathbb{N},
    \end{align*}
    for some $\xi_l\in [l,l+1].$ Since $\bar{\mathcal{C}}(K(t))$ converges, we have 
    \[\|\nabla \bar{\mathcal{C}}(K(\xi_l))\|^2=\|J^T(K(\xi_l))(L(K(\xi_l))-\hat{L})\|^2\rightarrow0.\] 
   By $\{\lambda_1(J^T(K(\xi_l)))\}_{l\in\mathbb{N}}$ we denote the minimum eigenvalues of $\{J^T(K(\xi_l))\}_{l\in \mathbb{N}}.$
Since $\{K(\xi_l)\}_{l\in \mathbb{N}}$ lies in a compact region of $\mathbb{R}^n$ and $J$ is positive definite, there exists $\lambda>0$ such that
\[
\lambda_1(J^T(K(\xi_l)))>\lambda,~~\forall l\in \mathbb{N}.
\]
So 
$$\|(L(K(\xi_l))-\hat{L})\|^2\le\frac{1}{\lambda^2}\|J^T(K(\xi_l))(L(K(\xi_l))-\hat{L})\|^2\rightarrow0,~~l\rightarrow\infty.$$
As a consequence, we have 
\[\|L(K^*)-\hat{L}\|=0.\]
Therefore by Theorem \ref{nie}, the prescribed total geodesic curvatures $\{\hat{L}_v\}_{v\in V}$ must satisfy \eqref{condition}.

Now we turn to prove \enquote{iii$\Rightarrow$i}.
We shall consider the potential function $\mathcal{E}$ mentioned in section \ref{sec:5}. If the condition \eqref{condition} holds, by Theorem \ref{nie} and Corollary \ref{proper}, the potential function $\mathcal{E}$ is proper. Now we compute
\begin{align*}
\ddt{\mathcal{E}(K(t))}&=\nabla\mathcal{E}^T\cdot \ddt{K},\\
&=-(L-\hat{L})^TJ^T(L-\hat{L}).
\end{align*}
Since $J^T$ is positive definite, $\ddt{\mathcal{E}(K(t))}\le0.$ Therefore $\mathcal{E}(K(t))$ is bounded from above. By Corollary \ref{proper}, $\mathcal{E}(K(t))$ is also bounded from below. Due to $\mathcal{E}$ is proper, $\{K(t):t\ge0\}$ must lie in a compact region of $\mathbb{R}^n.$ Let $\lambda_1(t)$ be the minimum eigenvalue of $J^T(K(t)).$  There exists $\hat{\lambda}>0$ such that 
\[
\lambda_1(t)>\hat{\lambda}, ~~~\forall t\ge0.
\]
Therefore,
\begin{align*}
    \ddt{\bar{\mathcal{C}}(K(t))}&=-\|J^T (L(K(t))-\hat{L})\|^2,\\
    &\le-\hat{\lambda}^2\|L(K(t))-\hat{L}\|^2,\\
    &=-2\hat{\lambda}^2\bar{\mathcal{C}}(K(t)).
\end{align*}
This indicates $0\le\bar{\mathcal{C}}(K(t))\le \bar{\mathcal{C}}(K(0))e^{-2\hat{\lambda}^2t}.$
Therefore $\|L(K(t))-\hat{L}\|^2$ converges exponentially fast to $0.$ And since $\ddt{K(t)}=-J^T(K(t))(L(K(t))-\hat{L})$ where $\|J^T(K(t))\|_2$ is uniformly bounded when $t\ge0$, $K(t)$ converges exponentially fast to some point as well. Finally, we finish the proof.
\end{proof}

\textbf{Acknowledgements.} The authors would like to thank Huabin Ge and Bobo Hua for helpful suggestions and useful comments on an earlier version of this paper. Lei is partially supported by NSFC, no.12122119. Zhou is partially supported by Shanghai Science and Technology Program [Project No. 22JC1400100].

\bibliographystyle{plain}
\bibliography{reference}

\noindent Ziping Lei, zplei@ruc.edu.cn\\[2pt]
\emph{School of Mathematics, Renmin University of China, Beijing, 100872, P.R. China} 
\\

\noindent Puchun Zhou, pczhou22@m.fudan.edu.cn\\[2pt]
\emph{School of Mathematical Sciences, Fudan University, Shanghai, 200433, P.R. China}
\end{document}